\documentclass[10pt]{amsart}
\usepackage[english]{babel}
\usepackage{amssymb}
\usepackage{amsmath}

\usepackage{lscape}
\usepackage{marvosym}

\newcommand{\can}{\overline{\phantom{x}}}

\newtheorem{dummy}{Dummy}

\newtheorem{lemma}[dummy]{Lemma}
\newtheorem{theorem}[dummy]{Theorem}
\newtheorem{proposition}[dummy]{Proposition}
\newtheorem{corollary}[dummy]{Corollary}

\theoremstyle{definition}

\newtheorem{example}[dummy]{Example}

\newcommand{\ignore}[1]{}

\author{Adam Owen}
\author{Susanne Pumpl\"un}

\email{owena004@gmail.com; pmzsp@nottingham.ac.uk}
\address{School of Mathematical Sciences\\
University of Nottingham\\ University Park\\ Nottingham NG7 2RD\\
United Kingdom }



\begin{document}

\title[The eigenspaces of twisted polynomials over cyclic field extensions]
{The eigenspaces of twisted polynomials over cyclic field extensions}

\begin{abstract}
Let $K$ be a field and $\sigma$ an automorphism of $K$ of order $n$.
 Employing a nonassociative algebra, we study the eigenspace of a bounded skew polynomial $f\in K[t;\sigma]$. We mainly treat the case that $K/F$ is a cyclic field extension  of degree $n$ with Galois group generated by $\sigma$.
  We obtain lower bounds on the dimension of the eigenspace, and compute it in special cases as a quotient algebra.
  Conditions under which a monic polynomial $f\in F[t]\subset K[t;\sigma]$ is reducible are obtained in special cases.
\end{abstract}

\maketitle

\section*{Introduction}

Let $D$ be a unital associative division ring and $R=D[t;\sigma,\delta]$ be a
skew polynomial ring, where $\sigma$ is an automorphism of $D$ and $\delta$ a left $\sigma$-derivation. Let $f\in R$ be a skew polynomial of degree $m>1$. The associative algebra
 $E(f)=\{g\in R\,|\, {\rm deg}(g)<m \text{ and }fg\in Rf\}$ is the \emph{eigenspace} of $f$.
 If $f$ is a bounded polynomial, then the nontrivial zero divisors in the eigenspace are in one-to-one correspondence with the irreducible factors of $f$ in $D[t;\sigma,\delta]$,  cf. for instance
 \cite{GLN13}. Therefore eigenspaces of skew polynomials regularly appear whenever skew polynomials are factorized, e.g.
in results on computational aspects of operator algebras, or in algorithms factoring skew polynomials over $\mathbb{F}_q(t)$ or over
$\mathbb{F}_q$, cf. \cite{giesbrecht1998factoring, GZ, GLN, G}.
For skew polynomial rings over local fields of positive characteristic, where the Brauer group is non-trivial, the irreducibility of a skew polynomial is equivalent to understanding a ring isomorphism to a full matrix ring over a field extension of the local field. This problem is not completely solved in the non-split case. Partial results have been obtained e.g. in \cite{G20}.

The eigenspace of $f$ also appears implicitly in classical constructions by  Amitsur  \cite{Am2, Am, Am3}, but was never recognized as the right nucleus of some nonassociative algebra.

In this paper, we investigate eigenspaces using a class of unital nonassociative algebras $S_f$ defined by Petit \cite{petit1966certains, P68}, which canonically generalize the quotient algebras $R/Rf$  obtained when factoring out a right invariant $f\in R$.
 The algebra
$S_f=D[t;\sigma,\delta]/D[t;\sigma,\delta]f$ is defined on the additive subgroup
$\{h\in R\,|\, {\rm deg}(h)<m \}$ of $R$ by using right division by $f$
to define the algebra multiplication via $g\circ h=gh \,\,{\rm mod}_r f $. Petit's algebras  were studied in detail in
\cite{petit1966certains, P68}, and for $K$ a finite  field (hence w.l.o.g. $\delta=0$) in \cite{LS}. Indeed, the algebra $S_f$ with $f(t)=t^2-i\in \mathbb{C}[t;\can]$, $\can$ the complex conjugation, already appeared in \cite{D06}
 as the first example of a nonassociative division algebra.
The right nucleus of $S_f$ is the eigenspace of $f$, if $f$ is not linear.
  Thus  the eigenspace of $f$ is an associative subalgebra of $S_f$.

We concentrate on the case that $R=K[t;\sigma]$, where $K/F$ is a cyclic Galois extension of degree $n$ with Galois group generated by  $\sigma$, and find conditions under which a monic polynomial $f\in R$ is reducible.

 In Section \ref{sec:prel},
we introduce our terminology and some results we need later.
In Section \ref{sec:powers of t} we determine when a power of $t$ lies in the right nucleus.
 This yields some lower bounds on the dimension of the right nucleus as an $F$-vector space.
These bounds can then later be used to decide if certain polynomials $f$ of degree $m$ which are not right invariant
are reducible. Let $f\in R$ be a bounded monic  polynomial that is not right invariant with ${\rm gcrd}(f,t)=1$, and
minimal central left multiple
$h(t)=\hat{h}(t^n)$, $\hat{h}(x) \in F[x]$ monic.
We show that for $f\in F[t]\subset K[t;\sigma]$, the quotient algebra
${\rm Nuc}(S_f)[t;\sigma ]/{\rm Nuc}(S_f)[t;\sigma ]f$
is a subalgebra of ${\rm Nuc}_r(S_f)$ (Theorem \ref{L_f subalgebra of the right nucleus}).
In particular if $f\in F[t]\subset K[t;\sigma]$ is bounded and we have
 ${\rm Nuc}_r(S_f)={\rm Nuc}(S_f)[t;\sigma]/{\rm Nuc}(S_f)[t;\sigma]f,$
 then  $f$ is irreducible in $R$, if and only if $f$ is irreducible in ${\rm Nuc}(S_f)[t;\sigma]$.
 In Section \ref{sec:elements in N} we look at the nucleus of $S_f$ for  $f\in R$.
Since ${\rm Nuc}(S_f) =  {\rm Nuc}_r(S_f) \cap K$, this helps us to understand which elements of $K$ lie in ${\rm Nuc}_r(S_f)$.
 In Section \ref{sec: h reducible} we assume only that $\hat{h}$ is irreducible in $F[x]$ and obtain some partial results for this case as well.
  In Section \ref{sec:low}, we look at the right nucleus of $S_f$ for low degree polynomials in $F[t]\subset K[t;\sigma]$, and in
 Section
\ref{sec:conclusion}, we  summarize for which types of skew polynomials which are not right invariant we can decide if they are reducible using our methods.

Note that cyclotomic extensions where $F=\mathbb{Q}$ and $K=\mathbb{Q}(\eta)$, with $\eta$ a primitive $p^n$th root of unity and $p$ prime, which have Galois group $\mathbb{Z}/p^n\mathbb{Z}$, and Kummer extensions $K=F(\sqrt[r]{a})$ of $F$, where $F$ contains a primitive $r$th root of unity $\mu$ and
$\sigma(\sqrt[r]{a})=\mu \sqrt[r]{a}$, are examples of skew polynomial rings that are employed in coding theory (e.g. in space-time block coding or for certain linear codes), where both reducible and irreducible $f$ are needed.

 This work is part of the first author's PhD thesis \cite{AO} written under the supervision of the second author.
 For more general results on eigenspaces of skew polynomials $f \in D[t;\sigma,\delta]$ the reader is referred to \cite{AO}.

%%%%%%%%%%%%%%%%%%%%%%%%%%%%%%%%%%%%%%%%%%%%%%%%%%%%%%%%%%%%%%%%%%%%%%%%%%%%%%%%%%%%%%%%%
%
%Preliminaries
%
%%%%%%%%%%%%%%%%%%%%%%%%%%%%%%%%%%%%%%%%%%%%%%%%%%%%%%%%%%%%%%%%%%%%%%%%%%%%%%%%%%%%%%%%%%

\section{Preliminaries} \label{sec:prel}

%%%%%%%%%%%%%%%%%%%%%%%%%%%%%%%%%%%%%%%%%%%%%%%%%%%%%%%%%%%%%%%%%%%%%%%%%%%%%%%%%%%%%%%%%%

\subsection{Nonassociative algebras} \label{subsec:nonassalgs}

%%%%%%%%%%%%%%%%%%%%%%%%%%%%%%%%%%%%%%%%%%%%%%%%%%%%%%%%%%%%%%%%%%%%%%%%%%%%%%%%%%%%%%%

Let $F$ be a field and let $A$ be an $F$-vector space. $A$ is an
\emph{algebra} over $F$ if there exists an $F$-bilinear map $A\times
A\to A$, $(x,y) \mapsto x \cdot y$, denoted simply by juxtaposition
$xy$, the  \emph{multiplication} of $A$. An algebra $A$ is
\emph{unital} if there is an element in $A$, denoted by 1, such that
$1x=x1=x$ for all $x\in A$. We will only consider unital algebras.

Associativity in $A$ is measured by the {\it associator} $[x, y, z] =
(xy) z - x (yz)$. The {\it left nucleus} of $A$ is defined as ${\rm
Nuc}_l(A) = \{ x \in A \, \vert \, [x, A, A]  = 0 \}$, the {\it
middle nucleus} of $A$ is ${\rm Nuc}_m(A) = \{ x \in A \, \vert \,
[A, x, A]  = 0 \}$ and  the {\it right nucleus} of $A$ is defined as
${\rm Nuc}_r(A) = \{ x \in A \, \vert \, [A,A, x]  = 0 \}$. ${\rm
Nuc}_l(A)$, ${\rm Nuc}_m(A)$, and ${\rm Nuc}_r(A)$ are associative
subalgebras of $A$. Their intersection
 ${\rm Nuc}(A) = \{ x \in A \, \vert \, [x, A, A] = [A, x, A] = [A,A, x] = 0 \}$ is the {\it nucleus} of $A$.
${\rm Nuc}(A)$ is an associative subalgebra of $A$ containing $F$
and $x(yz) = (xy)z$ whenever one of the elements $x, y, z$ lies in
${\rm Nuc}(A)$. Commutativity in $A$ is measured by the {\it commutator} $[x,y] = xy-yx$. The subspace of $A$ defined by ${\rm Comm}(A) = \{ x \in A : [x,y] = 0 \text{ for all } y \in A \}$ is called the {\it commutator} of $A$. The
 {\it center} of $A$ is $C(A)= {\rm Nuc}(A) \cap {\rm Comm}(A)$.

An $F$-algebra $A\not=0$ is called a \emph{division algebra} if for any
$a\in A$, $a\not=0$, both the left multiplication  with $a$, $L_a(x)=ax$,
and the right multiplication with $a$, $R_a(x)=xa$, are bijective.
If $A$ has finite dimension over $F$, $A$ is a division algebra if
and only if $A$ has no zero divisors \cite[ pp. 15, 16]{Sch}.

%%%%%%%%%%%%%%%%%%%%%%%%%%%%%%%%%%%%%%%%%%%%%%%%%%%%%%%%%%%%%%%%%%%%%%%%%

\subsection{Twisted polynomial rings $K[t;\sigma]$}

%%%%%%%%%%%%%%%%%%%%%%%%%%%%%%%%%%%%%%%%%%%%%%%%%%%%%%%%%%%%%%%%%%%%%%%%%%%%

Let $K$ be a field and $\sigma$ an automorphism of $K$ with fixed field $F={\rm Fix}(\sigma) = \{ a \in K : \sigma(a)=a \}$.
 The \emph{twisted polynomial ring} $R=K[t;\sigma]$
is the set of polynomials $a_0+a_1t+\dots +a_nt^n$ with $a_i\in K$,
where addition is defined term-wise and multiplication by $ta=\sigma(a)t $ for all $a\in K$.
For $f=a_0+a_1t+\dots +a_nt^n$ with $a_n\not=0$ define the {\it degree} of $f$ to be ${\rm
deg}(f)=n$, by convention ${\rm deg}(0)=-\infty$. Then ${\rm deg}(fg)={\rm deg}
(f)+{\rm deg}(g).$
 An element $f\in R$ of degree $m$ is \emph{irreducible} in $R$ if it is not a unit and  it has no proper factors,
 i.e if there do not exist $g,h\in R$ such that ${\rm deg}(g),{\rm deg} (h)<{\rm deg}(f)$ and $f=gh$.

$R$ is a left and right principal ideal domain and there is a right division algorithm in $R$: for all
$g,f\in R$, $g\not=0$, there exist unique $q,r\in R$ with ${\rm deg}(r)<{\rm deg}(f)$, such that $g=qf+r$ \cite[p.~3 and Proposition 1.1.14]{J96}.

 A twisted polynomial  $f \in R$ is \emph{bounded} if there exists a nonzero polynomial $f^* \in R$,  such that $Rf^*$ is the largest two-sided ideal of $R$ contained in $Rf^*$. The polynomial $f^*$ is uniquely determined by $f$ up to scalar multiplication by elements in $K^\times$. $f^*$ is called the \emph{bound} of $f$.
 The {\it left idealiser} of $f\in R$ is the set $I(f) = \lbrace g \in R \,|\, fg \in Rf \rbrace,$ which is the largest subring of $R$ within which $Rf$ is a two-sided ideal. The {\it eigenspace} of $f$ is the quotient ring $E(f) = I(f)/Rf= \lbrace g \in R \,|\, {\rm deg}(g) < m \text{ and } fg \in Rf \rbrace.$

%%%%%%%%%%%%%%%%%%%%%%%%%%%%%%%%%%%%%%%%%%%%%%%%%%%%%%%%%%%%%%%%%%%%%%%%%

\subsection{Nonassociative algebras obtained from twisted polynomial rings} \label{subsec:structure}

%%%%%%%%%%%%%%%%%%%%%%%%%%%%%%%%%%%%%%%%%%%%%%%%%%%%%%%%%%%%%%%%%%%%%%%%%%%

From now on, let $f \in R$ have positive degree $m$, and for $g \in R$ let $g\, {\rm mod}_r f$ denote the remainder of $g$ upon right division by $f$. The set $\{g\in R\,|\, {\rm deg}(g)<m\}$ endowed with the usual term-wise addition of polynomials and the multiplication $g \circ h = gh \, {\rm mod}_r f$ is a unital nonassociative ring $S_f$. We usually will simply use juxtaposition for the multiplication in $S_f$. $S_f$ is a unital nonassociative algebra over
$F_0 = \{ a \in K \,|\, ag = ga \text{ for all } g \in S_f \} = {\rm Comm}(S_f) \cap K.$
$F_0$ is a subfield of $K$ \cite{petit1966certains}. $S_f$ is called a {\it Petit algebra}. It can be easily seen that $F_0={\rm Fix}(\sigma)$, see \cite[pg.~6]{brown2018}. For all $a \in K^{\times}$ we have $S_f = S_{af}$, and if $f\in R$ has degree 1 then $S_f \cong K$.
In the following, we thus assume that $f$ is monic and that it has degree $m\geq 2$, unless specifically mentioned otherwise.
 $S_f$ is associative if and only if $f$ is \emph{right invariant}, i.e. $Rf$ a two-sided ideal in $R$.
In that case, $S_f$ is equal to the classical associative quotient algebra $R/(f)$. Note that  $f(t) = t^m - \sum\limits_{i=0}^{m-1}a_it^i \in R$ is right invariant if and only if $a_i \in F$ and $\sigma^m(d)a_i = a_i\sigma^i(d)$ for all $i \in \lbrace 0,1,\dots,m-1 \rbrace$ and for all $d \in K$
\cite[(15)]{petit1966certains}. In other words,
 $f$ is right invariant in $R$ if and only if $f(t)=g(t)t^n$ for some $g \in C(R)$ and some
integer $n\geq 0$  \cite[Theorem 1.1.22]{J96}.\\
If $S_f$ is not associative then ${\rm Nuc}_l(S_f)={\rm Nuc}_m(S_f)=K$
and $C(S_f)=F$.   Moreover,
$${\rm Nuc}_r(S_f)=\{g\in R\,|\, {\rm deg}(g)<m \text{ and }fg\in Rf\}.$$
 is the eigenspace of $f\in R$ \cite{petit1966certains}.

$S_f$ is a division algebra, if and only if $f$ is irreducible, if and only if ${\rm Nuc}_r(S_f)$ is a division algebra.  It is well known that each nontrivial zero divisor $q$ of $f$ in ${\rm Nuc}_r(S_f)$ gives a proper factor ${\rm gcrd}(q,f)$ of $f$, e.g. see \cite{GLN13}, where ${\rm gcrd}(q,f)$ denotes the greatest common right divisor of $q$ and $f$ in $R$.\\
If $f(t) \in F[t]\subset R$,
 then $F[t]/(f)$ is a commutative subring of ${\rm Nuc}_r(S_f),$ and a field extension of $F$ of degree $m$ if $f$ is irreducible as a polynomial in $F[t]$
 \cite[Proposition 2]{brown2018nonassociative}.

%%%%%%%%%%%%%%%%%%%%%%%%%%%%%%%%%%%%%%%%%%%%%%%%%%%%%%%%%%%%%%%%%%%%%%%%%%%%%%%%%%%%%%%%%%%%%%%%%%%%%%%%%%%%%

\subsection{The right nucleus of $S_f$ for irreducible $f$}\label{sec:first main results}

%%%%%%%%%%%%%%%%%%%%%%%%%%%%%%%%%%%%%%%%%%%%%%%%%%%%%%%%%%%%%%%%%%%%%%%%%%%%%%%%%%%%%%%%%%%%%%%%%%%%%%%%%%%%%

Throughout this section we assume that $\sigma$ has finite order $n>1$. Then $R$ has center $C(R) = F[t^n]\cong F[x]$, where $x=t^n$ \cite[Theorem 1.1.22]{J96}. \\
For any bounded $f \in R$ we  define the \emph{minimal central left multiple of $f$ in $R$} as the unique  polynomial of minimal degree $h \in F[t^n]$ such that $h = gf$ for some $g \in R$, and such that $h(t)=\hat{h}(t^n)$ for some monic $\hat{h} \in F[x]$. If the greatest common right divisor ${\rm gcrd}(f,t)$ of $f$ and $t$ is one, then $f^*\in C(R)$ \cite[Lemma 2.11]{GLN13}), and the minimal central left multiple of $f$ equals $f^*$ up to a scalar multiple from $K^\times$. From now on we therefore assume that $f$ is bounded with  $${\rm gcrd}(f,t)=1$$ and denote the minimal central left multiple  of $f$ by $h(t)=\hat{h}(t^n)$ with $\hat{h}(x) \in F[x]$ monic.

If $f$ is irreducible in $R$, then $\hat{h}(x)$ is irreducible in $F[x]$.
If $\hat{h}$ is irreducible in $F[x]$, then $h$ generates a maximal two-sided ideal $Rh$ in $R$ \cite[p.~16]{J96} and $f=f_1\cdots f_r$ for irreducible $f_i\in R$ such that $f_i\sim f_j$ for all $i,j$ (for a proof see \cite{TP21} or \cite{AO}).

The quotient algebra $R/Rh$ has the commutative $F$-algebra $C(R/Rh) \cong F[x]/ (\hat{h}(x))$ of dimension $deg(\hat{h})$ over $F$ as its center, cf. \cite[Lemma 4.2]{GLN13}.
Define $E_{\hat{h}}=F[x]/(\hat{h}(x))$. If $\hat{h}$ is irreducible in $F[x]$, then $E_{\hat{h}}$ is a field extension of $F$ of degree ${\rm deg}(\hat{h})$.

In \cite{LS}, Lavrauw and Sheekey determine the size of the right nucleus of $S_f$ for irreducible $f\in\mathbb{F}_{q^n}[t;\sigma]$, where $F=\mathbb{F}_q$ with $q=p^e$ for some prime $p$ and integer $e$.
  In this setting, $S_f$ is a  semifield of order $q^{mn}$ whenever $f$ is irreducible and not right invariant, and $|{\rm Nuc}_r(S_f)|=q^m$  \cite[Lemma 4]{LS}. This result generalizes as follows:

\begin{theorem}\label{thm:main2} (for the proof, cf. \cite{AO} or \cite{TP21})
Suppose that $f$ is irreducible. Let $k$ be the number of irreducible factors of $h$ in $R$.
 \\ (i) ${\rm Nuc}_r(S_f)$ is a central division algebra over $E_{\hat{h}}$ of degree $s=n/k$, and
 $$ R/Rh \cong M_k({\rm Nuc}_r(S_f)).$$
 In particular, this means that ${\rm deg}(\hat{h})=\frac{m}{s}$, ${\rm deg}(h)=\frac{nm}{s}$, and
 $[{\rm Nuc}_r(S_f) :F]=ms.$
 Moreover, $s$ divides $m$.
 \\
 (ii) If $n$ is prime and $f$ not right invariant, then ${\rm Nuc}_r(S_f)\cong E_{\hat{h}}.$
In particular, then $[{\rm Nuc}_r(S_f) :F]=m$, ${\rm deg}(\hat{h})=m$, and ${\rm deg}(h)=mn$.
\end{theorem}

Comparing vector space dimensions we obtain that $[S_f:{\rm Nuc}_r(S_f)]=k.$
Moreover, if ${\rm deg}(h)=mn$ and  $\hat{h}$ is irreducible in $F[x]$, then $f$ is irreducible and ${\rm Nuc}_r(S_f)\cong E_{\hat{h}}$ \cite[Proposition 4.1]{GLN13}.

\begin{corollary}\label{cor:main2}
Suppose that $f$ is irreducible and that $m$ is prime.
 Then $f$ is not right invariant and one of the following holds:
 \\ (i) ${\rm Nuc}_r(S_f)\cong E_{\hat{h}}$,
 $[{\rm Nuc}_r(S_f) :F]=m$, and ${\rm deg}(h)=mn$.
\\  (ii) ${\rm Nuc}_r(S_f)$ is a central division algebra over $F=E_{\hat{h}}$ of prime degree $m$, $[{\rm Nuc}_r(S_f) :F]=m^2$, and $m$ divides $n$. This case occurs when $\hat{h}(x)=x+a\in F[x]$, i.e.  $h(t)=t^n+a$.
\end{corollary}

\begin{proof}
Since $s$ divides $m$ by Theorem \ref{thm:main2}, $s=1$ which implies (i), or $s=m$. If $s=m$ then $\hat{h}$ has degree one and so $F=E_{\hat{h}}$. Furthermore, then ${\rm Nuc}_r(S_f)$ is a central simple algebra over $F$ degree $m$. Thus $f$ is not right invariant in both cases.
Since here we have ${\rm deg}(h)=km=n$, $m$ also must divide $n$ in this case.
\end{proof}

\begin{corollary}
 Suppose that $f$ is irreducible and not right invariant. Let $n=pq$ for $p$ and $q$ prime.
\\ (i) ${\rm Nuc}_r(S_f)\cong E_{\hat{h}}$ is a field extension of $F$ of degree $m$, or  ${\rm Nuc}_r(S_f)$ is a central division algebra over $E_{\hat{h}}$ of prime degree $q$ (resp., $p$),
 $[{\rm Nuc}_r(S_f):F] =qm$ (resp., $=pm$), and $q$ (resp., $p$) divides $m$.
 \\ (ii) If ${\rm gcd}(m,n)=1$, then ${\rm Nuc}_r(S_f)\cong E_{\hat{h}}$ is a field extension of $F$ of degree $m$.
\end{corollary}

\begin{proof}
(i) Since $f$ is not right invariant, we note that $k>1$. If $n=pq$ then the equation $n=ks$ %in the proof of Theorem \ref{thm:main2} (i)
forces either that $s=1$ and $k=n$, hence that ${\rm Nuc}_r(S_f)\cong E_{\hat{h}}$, or that $s\not=1$ and then w.o.l.o.g. that  $k=p$ and $s=q$, so that here
${\rm Nuc}_r(S_f)$ is a central division algebra over $E_{\hat{h}}$ of degree $q$, ${\rm deg}(h)=pm$, ${\rm deg}(\hat{h})=pm/pq=m/q,$ and $[{\rm Nuc}_r(S_f):F]=q^2\frac{m}{q}=qm$. In particular, $q$ divides $m$.
\\ (ii)
  If $m$ is not divisible by $p$ and $q$, then $s=1$ by the proof of (i), or else we obtain a contradiction.
\end{proof}

This observation generalizes as follows by induction:

\begin{corollary} \label{cor:main1}
Suppose that $f$ is irreducible and not right invariant. Let $n=p_1\cdots p_l$ be the prime decomposition of $n$.
\\ (i) ${\rm Nuc}_r(S_f)\cong E_{\hat{h}}$ is a field extension of $F$ of degree $m$, or
${\rm Nuc}_r(S_f)$ is a central division algebra over $E_{\hat{h}}$ of degree $q_1\cdots q_r$, with $q_i\in \{p_1,\dots, p_l\}$,
$[{\rm Nuc}_r(S_f):F] =q_1\cdots q_r m$, and $q_1\cdots q_r$  divides $m$.
 \\ (ii)  If ${\rm gcd}(m,n)=1$ (i.e., $m$ is not divisible by any set of prime factors of $n$), then ${\rm Nuc}_r(S_f)\cong E_{\hat{h}}$ is a field extension of $F$ of degree $m$.
\end{corollary}

\begin{corollary}\label{thm:main3}
Let $f \in F[t] \subset R$. Suppose that $f$ is irreducible in $R$ and not right invariant. Let $n$ either be prime or ${\rm gcd}(m,n)=1$.
Then ${\rm Nuc}_r(S_f)\cong F[t]/(f).$
\end{corollary}

\begin{proof}
If $f$ is irreducible in $R$, then $F[t]/(f)$ is a subfield of the right nucleus of degree $m$, hence must be all of the right nucleus, since that has dimension $m$ due to our assumptions (Theorem \ref{thm:main2} (ii), Corollary \ref{cor:main1} (ii)).
\end{proof}

%%%%%%%%%%%%%%%%%%%%%%%%%%%%%%%%%%%%%%%%%%%%%%%%%%%%%%%%%%%%%%%%%%%%%%%%%%%%%%%%%%%%%%%%%%%%%%%%%%%%%%%%%%%%%%
%
% The elements of $K$ that lie in the right nucleus
%
%%%%%%%%%%%%%%%%%%%%%%%%%%%%%%%%%%%%%%%%%%%%%%%%%%%%%%%%%%%%%%%%%%%%%%%%%%%%%%%%%%%%%%%%%%%%%%%%%%%%%%%%%%%%%%%
%%%%%%%%%%%%%%%%%%%%%%%%%%%%%%%%%%%%%%%%%%%%%%%%

\section{Powers of $t$ that lie in the right nucleus of $S_f$} \label{sec:powers of t}

Throughout this section, let $f(t) = t^m - \sum\limits_{i=0}^{m-1} a_it^i \in R=K[t;\sigma]$ be not right invariant. Initially, we do not assume anything on the ring $R$.

\begin{theorem} \label{Petit(5)} \cite{petit1966certains}
 The following are equivalent:
\\ (i)  $a_i \in {\rm Fix}(\sigma)$ for all $i \in \lbrace 0,1,\dots,m-1 \rbrace$,
\\ (ii) $t \in {\rm Nuc}_r(S_f)$,
\\ (iii) $t^mt = tt^m$,
\\ (iv) $ft \in Rf$.
\\ (v) all powers of $t$ are associative in $S_f$.
\end{theorem}

\begin{proof} (i) and (ii) are equivalent by \cite[(16)]{petit1966certains} and (ii), (iii), (iv) and (v) are equivalent by \cite[(5)]{petit1966certains}.
\end{proof}

We obtain the following weak generalization of Theorem \ref{Petit(5)}:

\begin{theorem}\label{Condition for tk in the right nucleus}
Let $k \in \lbrace 1,2,\dots,m-1 \rbrace$. If $a_i \in {\rm Fix}(\sigma^k)$ for all $i \in \lbrace 0,1,\dots,m-1 \rbrace$, then $t^k \in {\rm Nuc}_r(S_f)$. In particular, then $t^mt^k = t^kt^m$ in $S_f$.
\end{theorem}

\begin{proof}
Suppose that $a_i \in {\rm Fix}(\sigma^k)$ for all $i$. Then
\begin{align*}
ft^k &= (t^m-\sum\limits_{i=0}^{m-1} a_it^i)t^k
     = t^mt^k - \sum\limits_{i=0}^{m-1} a_it^it^k \\
     &= t^kt^m - t^k\sum\limits_{i=0}^{m-1} \sigma^{-k}(a_i)t^i
     = t^k(t^m - \sum\limits_{i=0}^{m-1} a_it^i) \text{ (as $a_i \in {\rm Fix}(\sigma^k)$ $\forall i$) }\\
     &= t^kf \in Rf,
\end{align*}
i.e. $ft^k \in Rf$, and so $t^k \in {\rm Nuc}_r(S_f)$ as claimed. Since $t^k \in {\rm Nuc}_r(S_f)$, we have in particular that $[t^k,t^{m-k},t^k]=0$ in $S_f$, that is $t^k(t^{m-k}t^k)=(t^kt^{m-k})t^k$. Therefore $t^kt^m=t^mt^k$ in $S_f$.
\end{proof}

From now on we often write $N = {\rm Nuc}(S_f)$ for ease of notation.

\begin{proposition}\label{t^s in Nuc}
Suppose that there exists $s \in \lbrace 1,2,\dots,m-1 \rbrace$ such that $f \in {\rm Fix}(\sigma^s)[t;\sigma]$.
\\ (i) If $m=qs$ for some positive integer $q$, then
$$N \oplus  Nt^s \oplus  Nt^{2s} \oplus \dots \oplus  Nt^{(q-1)s}\oplus  N (\sum_{i=0}^{m-1} a_it^i) $$
is an $F$-sub vector space of ${\rm Nuc}_r(S_f).$
\\ (ii) If $m = qs + r$ for some positive integers $q,r$ with $0 < r < s$, then $$N \oplus Nt^s \oplus Nt^{2s} \oplus \dots \oplus Nt^{qs}$$
is an $F$-sub vector space of ${\rm Nuc}_r(S_f).$
\end{proposition}

\begin{proof}
 (i) Since $a_i \in {\rm Fix}(\sigma^s)$, we have that $t^s \in {\rm Nuc}_r(S_f)$ by Theorem \ref{Condition for tk in the right nucleus}. Since the right nucleus is a subalgebra of $S_f$, this implies that $t^{2s},\dots,t^{(q-1)s}, (t^s)^q = t^m = \sum_{i=0}^{m-1} a_it^i  \in {\rm Nuc}_r(S_f)$. Furthermore, we know that $N \subset  {\rm Nuc}_r(S_f)$, and so
$Nt^{js} \subset  {\rm Nuc}_r(S_f)$ for any $j \in \lbrace 0,1,\dots,q \rbrace$. Therefore
   $N \oplus Nt^s \oplus \dots \oplus Nt^{(q-1)s} \oplus N(\sum_{i=0}^{m-1} a_it^i)\subset {\rm Nuc}_r(S_f)$
    as claimed.
\\ (ii) We have $t^s \in {\rm Nuc}_r(S_f)$. Again since ${\rm Nuc}_r(S_f)$ is a subalgebra of $S_f$, this implies that $t^{2s},\dots t^{qs}, t^{(q+1)s},\dots \in {\rm Nuc}_r(S_f)$, hence the assertion as in (i).
\end{proof}

Note that the powers $t^{qs}, t^{(q+1)s}, t^{(q+2)s},\dots$ of $t^s$ in Proposition \ref{t^s in Nuc} (ii) lie in ${\rm Nuc}_r(S_f)$, but they need not be equal to polynomials in $t^s$, since $qs,(q+1)s, (q+2)s,\dots \geq m$.

\begin{corollary}
Let $K/F$ be a cyclic Galois extension of degree $n<m$ with Galois group ${\rm Gal}(K/F)=\langle \sigma \rangle$.
\\ (i) If $m=qn$, then
$$N \oplus  Nt^n \oplus Nt^{2n} \oplus \dots \oplus Nt^{(q-1)n}$$
is an $F$-sub vector space of ${\rm Nuc}_r(S_f)$ of dimension $q[N:F]$ and
$t^m = \sum_{i=0}^{m-1} a_it^i \in {\rm Nuc}_r(S_f).$
\\(ii) If $m = qn + r$ for some positive integers $q,r$ with $0 < r < n$, then
$$N \oplus Nt^n \oplus Nt^{2n} \oplus \dots \oplus Nt^{qn} $$
is an $F$-sub vector space of ${\rm Nuc}_r(S_f)$ of dimension $(q+1)[N:F]$.
 In particular, if  $n$ is either prime or ${\rm gcd}(m,n)=1$, ${\rm gcrd}(f,t)=1$, as well as $[N:F]= n$, then $f$ is reducible.
\end{corollary}

\begin{proof}
There exist integers $q,r$ such that $q \neq 0$, and $m=qn+r$ where $0 \leq r< n$. Moreover, we have
  $a_i \in {\rm Fix}(\sigma^n)  = K$ for all $i \in \lbrace 0,1,\dots,m-1 \rbrace$ for
  every $f(t) = t^{m} - \sum_{i=0}^{m-1} a_it^i \in R$. By Theorem \ref{Condition for tk in the right nucleus}
   this yields the assertions.
\end{proof}

We write $\sigma=\sigma\vert_{{\rm Fix}(\sigma^c)}$, for ease of notation, then:

\begin{theorem}
\label{L_f subalgebra of the right nucleus}
Suppose that $f(t) \in F[t]\subset R$.
Then
$N[t;\sigma]/N[t;\sigma]f$
is a Petit algebra and an associative subalgebra of ${\rm Nuc}_r(S_f)$.
\end{theorem}

\begin{proof}

Clearly $N[t;\sigma|_{N}]$ is well-defined, $f(t) \in F[t]\subset N[t;\sigma]$, and so  $N[t;\sigma]/N[t;\sigma]f$ is a subalgebra of $S_f$.

 Now $N \subset  {\rm Nuc}_r(S_f)$, and since $a_i \in F$ for all $i$, we have that $t^j \in {\rm Nuc}_r(S_f)$ for all $j$ by Theorem \ref{Condition for tk in the right nucleus}. Thus
$N \oplus Nt\oplus \dots \oplus Nt^{m-1} \subset  {\rm Nuc}_r(S_f)$ is contained in the right nucleus.
We have proved the assertion.
\end{proof}

\begin{corollary} 
 Suppose that $f(t)  \in F[t]\subset R$ is bounded and that $${\rm Nuc}_r(S_f)=N[t;\sigma]/N[t;\sigma]f.$$ Then $f$ is irreducible in $R$, if and only if $f$ is irreducible in $N[t;\sigma]$.
\end{corollary}

 \begin{proof}
 Let $f$ be irreducible in $N[t;\sigma]$, then $N[t;\sigma]/N[t;\sigma]f={\rm Nuc}_r(S_f)$ is a division algebra and therefore $f$ is irreducible in $R$.
 \end{proof}

%%%%%%%%%%%%%%%%%%%%%%%%%%%%%%%%%%%%%%%%%%%%%%%%%%%%%%%%%%%%%%%%%%%%%%%%%%%%%%%%%%%%%%%%%%%%%%%%%%%%%%%%%%%%%%%

\section{The nucleus of $S_f$}\label{sec:elements in N}

%%%%%%%%%%%%%%%%%%%%%%%%%%%%%%%%%%%%%%%%%%%%%%%%%%%%%%%%%%%%%%%%%%%%%%%%%%%%%%%%%%%%%%%%%%%%%%%%%%%%%%%%%%%%%%%

In this section we again assume that $f$ is not right invariant.
Then the elements of $K$ which lie in ${\rm Nuc}_r(S_f)$ are exactly the elements in the nucleus of $S_f$:

\begin{lemma}
$K \cap {\rm Nuc}_r(S_f)={\rm Nuc}(S_f)$.
\end{lemma}

\begin{proof}
Since $f$ is not right invariant, $S_f$ is not associative and thus ${\rm Nuc}_l(S_f) = {\rm Nuc}_m(S_f) = K$. Therefore
${\rm Nuc}(S_f)  = {\rm Nuc}_l(S_f) \cap {\rm Nuc}_m(S_f) \cap {\rm Nuc}_r(S_f)= K \cap {\rm Nuc}_r(S_f) $.
\end{proof}

Clearly $F \subset  {\rm Nuc}(S_f)$. 
Let
$f(t) = t^m - \sum_{i=0}^{m-1}a_it^i \in R.$

\begin{theorem}\label{Coeffs}
${\rm Nuc}(S_f)=\lbrace b \in K \,|\, \sigma^m(b)a_i=a_i\sigma^i(b) \text{ for all } i =0,1,2,\dots,m-1 \rbrace.$
\end{theorem}

\begin{proof}
Let $c\in \lbrace b \in K \,|\, \sigma^m(b)a_i=a_i\sigma^i(b) \text{ for all } i =0,1,2,\dots,m-1 \rbrace$. Then
an easy calculation shows that $f(t)c \in Rf$, hence that
$c\in {\rm Nuc}_r(S_f)=\{g\in R\,|\, {\rm deg}(g)<m \text{ and }fg\in Rf\}$.

Conversely, let $c \in {\rm Nuc}(S_f)= {\rm Nuc}_r(S_f) \cap K$. Then $[a(t),b(t),c]=0$ for all $a(t),b(t) \in S_f$, in particular,  $[t^k,t^{m-k},c]=0$ for all $k \in \lbrace 1,2,\dots,m-1\rbrace$. This implies $(t^kt^{m-k})c = t^k(t^{m-k}c),$ hence
$$ (\sum\limits_{i=0}^{m-1}a_it^i)c = t^k(\sigma^{m-k}(c)t^{m-k})\Rightarrow \sum\limits_{i=0}^{m-1}a_i\sigma^i(c)t^i = \sigma^m(c) \sum\limits_{i=0}^{m-1}a_it^i,$$
and thus $ a_i \sigma^i(c) = \sigma^m(c)a_i \text{ for each } i=0,1,\dots,m-1.$
Therefore $c \in \lbrace b \in K : \sigma^m(b)a_i=a_i\sigma^i(b) \text{ for all } i =0,1,2,\dots,m-1 \rbrace$ as required.
\end{proof}

 We denote the indices of the nonzero coefficients $a_i$ of $f(t) = t^m - \sum\limits_{i=0}^{m-1} a_it^i \in R$ by $\lambda_1,\dots,\lambda_{r}$, $1 \leq r \leq m$. The set of these indices we call
 $\Lambda_f = \{ \lambda_1,\lambda_2,\dots,\lambda_{r} \}\subset \lbrace 0,1,\dots,m \rbrace,$
 and write $\Lambda = \Lambda_f$ when it is clear from context which $f$ is being used.

\begin{proposition}\label{Lsigma,f for a field extension} 
 (i)
 ${\rm Nuc}(S_f) = \bigcap\limits_{j=1}^{r} {\rm Fix}(\sigma^{m-\lambda_j}).$
 In particular, ${\rm Nuc}(S_f)$ is a subfield
of $K$.
\\ (ii) If $a_{m-1} \neq 0$, then ${\rm Nuc}(S_f)=F.$
 \end{proposition}

\begin{proof}
 (i) Let $u\in {\rm Nuc}(S_f) = \lbrace u \in K \,|\, \sigma^m(u)a_i = a_i\sigma^i(u) \text{ for all } i \in \lbrace 0,1,\dots,m-1 \rbrace \rbrace.$
 Then $\sigma^m(u)a_i = a_i\sigma^i(u) $ for each $i$ if and only if $ \sigma^{m-i}(u) = u$  for each $i$ such that $ a_i \neq 0$, which is equivalent to $\sigma^{m-\lambda_j}(u) = u \text{ for each } \lambda_j \in \Lambda_f$. This yields the assertion.
 \\ (ii) 
  Let $u\in {\rm Nuc}(S_f)$, then $\sigma^m(u)a_{m-1} = a_{m-1}\sigma^{m-1}(u)$ yields $\sigma(u)=u$, hence $u\in F$.
This implies the assertion.
\end{proof}

\begin{example} Let $\mathbb{F}_{16}=\mathbb{F}(a)$ with $a^4=a+1$ and $K=\mathbb{F}_{16}(z)$ be the rational function field over $\mathbb{F}_{16}$. Define $\sigma:K\longrightarrow K$, $\sigma(t)=a^5t$, then $\sigma$ has order 3 and $F={\rm Fix}(\sigma)=\mathbb{F}_{16}(z^3)={\rm Fix}(\sigma^2)$. Let $R=K[t;\sigma]$, then $C(R)=F[t^3]$ \cite[Example 2.16]{GLN13}.
Note that not every $f$ is bounded in this setup.

Let $f\in R$ be monic of degree $m$, then ${\rm Nuc}(S_f) = \bigcap\limits_{j=1}^{r} {\rm Fix}(\sigma^{m-\lambda_j})$
(Proposition \ref{Lsigma,f for a field extension}). If  we have $m-i=3l$ for all $a_{i} \neq 0$ then $N = K$, else
 ${\rm Nuc}(S_f)= F$.
\\ (i)
 Suppose that $f$ has degree $m=3q\geq 4$, then ${\rm Nuc}_r(S_f)$ contains an $F$-vector space of dimension $q[N:F]$.
If $f = g(t^3)$ for some $g \in K[x]$ then ${\rm Nuc}(S_f)=K$ and
$ K[x]/(g(x))=K \oplus K t^3 \oplus K t^{6} \oplus \dots \oplus Kt^{3(q-1)}$ is a sub vector space of ${\rm Nuc}_r(S_f)$.
\\ (ii)
Let $f(t)=t^2+\frac{1}{t+a}t+az^2+1$, then $f^*(t)=t^6+\frac{(a^3+a)z^3+a^2+a+1}{a^2t^3+a^2+a}t^3+a^3z^6+1$
 \cite[Example 2.16]{GLN13}, so $f$ is bounded and $f^*\in C(R)$. Here
$$\hat{h}(x)= x^2+\frac{(a^3+a)z^3+a^2+a+1}{a^2t^3+a^2+a}x+a^3z^6+1\in F[x]$$
has degree 2, and $h$ has degree $6=mn$.
Therefore $f$ is irreducible  and
$${\rm Nuc}_r(S_f)\cong F[x]/(x^2+\frac{(a^3+a)z^3+a^2+a+1}{a^2t^3+a^2+a}x+a^3z^6+1)$$
by Theorem \ref{thm:main2}.
\end{example}

From now on unless specified otherwise let $K/F$ be a cyclic Galois extension of degree $n>1$ with ${\rm Gal}(K/F)=\langle \sigma\rangle$.
Then $R$ has center $C(R) = F[t^n]\cong F[x]$, where $x=t^n$ \cite[Theorem 1.1.22]{J96} and every $f\in R$ is bounded.

\begin{theorem}\label{ConditionforL=F}
 If $d = {\rm gcd}(m-\lambda_1,m-\lambda_2,\dots,m-\lambda_r,n)$, then
$${\rm Nuc}(S_f) = {\rm Fix}(\sigma^d),$$
that is  $[{\rm Nuc}(S_f):F]  = d$.
 In particular, $ {\rm Nuc}(S_f)= F$  if and only if $d=1$.
\begin{proof}
By Proposition \ref{Lsigma,f for a field extension}, we have
$${\rm Nuc}(S_f) = \bigcap\limits_{\lambda_j \in \Lambda} {\rm Fix}(\sigma^{m-\lambda_j}) = {\rm Fix}(\sigma^{m-\lambda_1}) \cap {\rm Fix}(\sigma^{m-\lambda_2}) \cap \dots \cap {\rm Fix}(\sigma^{m-\lambda_r}).$$
It follows immediately that ${\rm Nuc}(S_f) = {\rm Fix}(\sigma^d)$. Clearly ${\rm Nuc}(S_f)=F$ if and only if ${\rm Fix}(\sigma^d) = F$ if and only if $\langle \sigma^d \rangle = \langle \sigma \rangle$, which is true if and only if $\sigma^d$ has order $n$. Now $ord(\sigma^d) = \frac{n}{{\rm gcd}(n,d)} = \frac{n}{d}= n$ if and only if $d =1$.
\end{proof}
\end{theorem}

\begin{corollary}\label{prime degree cyclic Nuc calc}
Let $K/F$ have prime degree $p$. Then
 ${\rm Nuc}(S_f) =K$ if and only if $m-\lambda_j$ is a multiple of $p$ for all $\lambda_j \in \Lambda$.
In other words, ${\rm Nuc}(S_f) = F$ if and only if there exists $\lambda_j \in \Lambda$ such that $m-\lambda_j$ is not divisible by $p$.
\end{corollary}

\begin{proof}
We have ${\rm Nuc}(S_f) =K$ if and only if $[{\rm Nuc}(S_f):F]=p$, i.e. if and only if $d = p$. Now $d = {\rm gcd}(m-\lambda_1,m-\lambda_2,\dots,m-\lambda_r,p)=p$ if and only if $m-\lambda_j$ is a multiple of $p$ for all $\lambda_j \in \Lambda$. Since second assertion is equivalent to the first the result follows immediately.
\end{proof}

\begin{theorem}
Let $K/F$ be  of degree $n=bc<m$ for some $b \in \mathbb{N}$.
If  $[{\rm Nuc}(S_f):F]=c$ then $m=qc+r$ for some integers $q,r$ with $0\leq r < c$, and $f(t) = g(t^c)t^r$, where $g$ is a polynomial of degree $q$ in $K[t^c;\sigma^c]$.
\end{theorem}

\begin{proof}
By Theorem \ref{ConditionforL=F}, we have that ${\rm Nuc}(S_f) = {\rm Fix}(\sigma^d)$, where $d = {\rm gcd}(m- \lambda_1,m-\lambda_2,\dots,m-\lambda_r,n)$. Now $d=c$ if and only if $m-\lambda_j$ is a multiple of $c$ for all $\lambda_j \in \Lambda$.  But $m-\lambda_j$ is equal to a multiple of $c$ if and only if $\lambda_j = r + cl$ for some integer $l$ such that $0 \leq l < q$ (since $m=qc+r$). Therefore we obtain $\Lambda \subset  \lbrace r, r+c,r+2c,\dots,r+(q-1)c \rbrace.$ Thus
\begin{align*}
f(t) &= t^{qc+r} - a_{(q-1)c+r}t^{(q-1)c+r} - \dots - a_{r+c}t^{r+c} - a_rt^r\\
&= [(t^c)^{q} - a_{(q-1)c+r}(t^c)^{(q-1)} - \dots - a_{r+c}t^c - a_r]t^r= g(t^c)t^r
\end{align*}
where $g$ has degree $q$ in $K[t^c;\sigma^c]$.
\end{proof}

\begin{example} Let $K=\mathbb{Q}(\zeta)$, $\sigma:K\longrightarrow K,$ $\sigma(\zeta)=\zeta^2$, and $R$. Then $\sigma$ has order three, $F={\rm Fix}(\sigma)=\mathbb{Q}(\zeta^4+\zeta^2+\zeta)$, and $C(R)={\rm Fix}(\sigma)[x]$ with $x=t^3$. Every $f\in R$ is bounded. Moreover, $[K:F]=3$ and $[\mathbb{Q}(\zeta^4+\zeta^2+\zeta):\mathbb{Q}]=2$.
Let  $f\in \mathbb{Q}(\zeta)[t,\sigma]$ be monic of degree $m$, then
${\rm Nuc}(S_f) = \bigcap\limits_{j=1}^{r} {\rm Fix}(\sigma^{m-\lambda_j})\in \{K,F\}$ (Proposition \ref{Lsigma,f for a field extension}).
If  we have $m-i=3l$ for all $a_{i} \neq 0$ then ${\rm Nuc}(S_f)= K$,
else ${\rm Nuc}(S_f) = F$.
\\ (i)
Suppose that $f$ has degree $m=3q\geq 4$, then ${\rm Nuc}_r(S_f)$ contains a $F$-sub vector space
  of dimension $q[{\rm Nuc}(S_f):F]$.
If $f (t)= g(t^3)$ for some $g \in K[x]$, then ${\rm Nuc}(S_f)=K$ and $K[x]/(g(x))\cong K \oplus K t^3 \oplus K t^{6} \oplus \dots \oplus Kt^{3(q-1)}$ is an
$F$-sub vector space of ${\rm Nuc}_r(S_f)$. If this $f$ is also irreducible and not right invariant, then $a_0\not=0$,
and $[{\rm Nuc}_r(S_f):F]=m$.
Thus in this case
$${\rm Nuc}_r(S_f)\cong K[x]/(g(x)). $$
 (ii)
Suppose that $f\in \mathbb{Q}(\zeta^4+\zeta^2+\zeta)[t]$ is not right invariant and we have $m-i\not=3l$ for some $a_{i} \neq 0$. Then
$${\rm Nuc}(S_f)/{\rm Nuc}(S_f)f= \mathbb{Q}(\zeta^4+\zeta^2+\zeta)[t]/(f(t))\subset {\rm Nuc}_r(S_f).$$
In particular, if $f$ is irreducible in $R$ then
$${\rm Nuc}_r(S_f)= \mathbb{Q}(\zeta^4+\zeta^2+\zeta)[t]/(f(t)).$$
 \end{example}

%%%%%%%%%%%%%%%%%%%%%%%%%%%%%%%%%%%%%%%%%%%%%%%%%%%%%%%%%%%%%%%%%%%%%%%%%%%
%
% The case where only $\hat{h}(x)$ is irreducible
%
%%%%%%%%%%%%%%%%%%%%%%%%%%%%%%%%%%%%%%%%%%%%%%%%%%%%%%%%%%%%%%%%%%%%%%%%%%%%%

%%%%%%%%%%%%%%%%%%%%%%%%%%%%%%%%%%%%%%%%%%%%%%%%%%%%%%%%%%%%%%%%%%%%%%%%%%%%%

\section{The case that only $\hat{h}(x)$ is irreducible  in $F[x]$}\label{sec: h reducible}

%%%%%%%%%%%%%%%%%%%%%%%%%%%%%%%%%%%%%%%%%%%%%%%%%%%%%%%%%%%%%%%%%%%%%%%%%%%%%

In this section we assume that $\sigma$ has finite order $n>1$, $f$ is bounded and that $\hat{h}$ is irreducible in $F[x]$. Then $f=f_1\cdots f_l$ for irreducible $f_i\in R$ such that $f_i\sim f_j$ for all $i,j$ (\cite{AO}, cf. \cite{TP21}). Let ${\rm deg}(f_i)=r $, then $m=rl$, and let $k$ be the number of irreducible factors of $h$ in $R$ (then $l\leq k$).

\begin{theorem}\label{thm:mainII}
For every $i$, $1\leq i \leq l$, $E(f_i)$ is a central division algebra over $E_{\hat{h}}$ of degree $s^{\prime}=n/k$ and
$$R/Rh \cong M_k(E(f_i)),\quad {\rm Nuc}_r(S_f) \cong M_l(E(f_i)).$$
 In particular, ${\rm Nuc}_r(S_f)$ is a central simple algebra over $E_{\hat{h}}$ of degree $s = ls^{\prime}$, $deg(\hat{h})=\frac{r}{s^{\prime}}=\frac{m}{s}$, $deg(h) = \frac{rn}{s^{\prime}} = \frac{mn}{s}$, and
 $$[{\rm Nuc}_r(S_f) :F] = l^2rs^{\prime} = ms.$$
 Moreover,  $s^{\prime}$ divides ${\rm gcd}(r,n)$, and $s$ and $l$ divide ${\rm gcd}(m,n)$.
\end{theorem}

\begin{proof}
Since $h$ is a two-sided maximal element in $R$, the irreducible factors $h_i$ of any factorization $h=h_1\cdots h_k$ of $h$ in $R$ are all similar. Now  $h(t)=p(t)f(t)$ for some $p(t)\in R$ and so comparing the irreducible factors of $f$ and $h$ and employing \cite[Theorem 1.2.9]{J96}, we see that $f=f_1\cdots f_l$ for irreducible $f_i\in R$ such that $f_i\sim f_j$ for all $i,j$ (and also $f_i\sim h_j$ for all $i,j$), with $l\leq k$.
In particular, $R/Rf_i \cong R/Rf_j$ for all $i,j$.
Moreover, $R/Rh \cong M_k(E(f_i))$ is a simple Artinian ring \cite[Theorem 1.2.19]{J96}.
 Each of the polynomials $h_i$, resp., $f_i$, has minimal central left multiple $h$ \cite[Proposition 5.2]{GLN13}.
Let $A=R/Rh$.
We obtain
 $$R/Rf \cong R/Rf_1 \oplus R/Rf_2 \oplus \cdots \oplus R/Rf_l$$
as a direct sum of simple left $A$-modules (e.g. see \cite[Corollary 4.7]{GLN13}).
Let $g$ be an irreducible factor of $h$ in $R$. Since $R/Rf_i \cong R/Rg$, we get
$R/Rf \cong (R/Rg)^{\oplus l}$
as left $A$-modules.
By \cite[Exercise 6.7.2, Lemma 6.7.5]{sullivan2004} we have
$${\rm End}_A(R/Rf) \cong {\rm End}_A((R/Rg)^{\oplus l}) \cong M_l({\rm End}_A(R/Rg))$$
as rings.

Since $h$ is the minimal central left multiple of $f$ and of $g$,  $Rh = {\rm Ann}_R(R/Rf) = {\rm Ann}_R(R/Rg)$ \cite[pg.~38]{jacobson1943theory}, hence ${\rm End}_R(R/Rf)={\rm End}_A(R/Rf)$, ${\rm End}_R(R/Rg) = {\rm End}_A(R/Rg)$, %by Lemma \ref{Lemma 3},
and
$${\rm End}_R(R/Rf) \cong M_l({\rm End}_R(R/Rg)).$$
 Finally,   $E(g)\cong {\rm End}_R(R/Rg)$, therefore
$$E(f) \cong  M_l(E(g)).$$
Since $g$ is irreducible of degree $r$ with minimal central left multiple $h(t) = \hat{h}(t^n)$, $E(g)$ is a central division algebra over $E_{\hat{h}}$ of degree $s^{\prime}=n/k$, where $k$ is the number of irreducible divisors of $h$ in $R$, ${\rm deg}(\hat{h})=\frac{r}{s^{\prime}}=\frac{m}{s}$ and ${\rm deg}(h)=\frac{rn}{s^{\prime}}=\frac{mn}{s}$ by Theorem \ref{thm:main2}. Finally, since $E(f) \cong M_l(E(g)),$
$E(f)$ is a central simple algebra over $E_{\hat{h}}$ of degree $s = ls^{\prime}$, and $[E(f) : F] = s^2{\rm deg}(\hat{h}) = ms.$
The assertion follows since $E(f_i)=E(g)$.

Now
$s^{\prime}=n/k$, and ${\rm deg}(\hat{h})= r/s^{\prime}$, i.e. $s^{\prime}$ divides both $n$ and $r$, hence $s^{\prime}$ divides ${\rm gcd}(n,r)$.
Next, $s$ divides ${\rm gcd}(m,n)$:
${\rm deg}(\hat{h})=m/s $ means  $s$ divides $m$. Also $[S_f:F] = b[{\rm Nuc}_r(S_f):F]$ for some positive integer $b$. We know that $[S_f:F]=mn$  and that $[{\rm Nuc}_r(S_f):F]=ms$, hence $mn=bms$. Cancelling $m$ yields $n=bs$, i.e. $s$ divides $n$. The result follows immediately.
Finally, $l$ divides ${\rm gcd}(m,n)$:
Since $s=ls^{\prime}$, $l$ divides $s$. Hence $l$ divides ${\rm gcd}(m,n)$ by the above.
\end{proof}

Comparing $F$-vector space dimensions, we  obtain that $[S_f:{\rm Nuc}_r(S_f)]=k/l.$

\begin{corollary}\label{cor:new}
Suppose that $\hat{h}(x)$ is irreducible in $F[x]$.
\\
 (i) If $m$ is prime, then one of the following holds:
\\ (a) ${\rm Nuc}_r(S_f) \cong E_{\hat{h}}$ is a field extension of $F$ of degree $m$,
\\ (b) ${\rm Nuc}_r(S_f)$ is a central division algebra over $F$ of degree $m$,
\\ (c) ${\rm Nuc}_r(S_f) \cong M_m(F)$.
\\ (ii) If ${\rm gcd}(m,n)=1$, or $n$ is prime and $f$ not right invariant, then $f$ is irreducible and
${\rm Nuc}_r(S_f) \cong E_{\hat{h}}$
is a field extension of $F$ of degree $m={\rm deg}(\hat{h})$, and ${\rm deg}(h)=mn$.
 \end{corollary}

\begin{corollary}\label{cor: L_f subalgebraII}
 Suppose that $f \in F[t]\subset R$ is not right invariant, and that $\hat{h}(x)$ is irreducible in $F[x]$.
 \\ (i) If $\hat{h}(x)$ is irreducible and $[N : F]=ln/k$, then
  ${\rm Nuc}_r(S_f)= N[t;\sigma]/N[t;\sigma]f.$
  \\ (ii) If $[N:F]>\frac{nl}{k}$, then $\hat{h}(x)$ is reducible, and therefore $f$ as well.
 \end{corollary}

\begin{proof}
(i) We know that $N[t;\sigma]/N[t;\sigma]f(t)$
is a subalgebra of ${\rm Nuc}_r(S_f)$  of dimension $\frac{lnm}{k}$ over $F$ (Theorem \ref{L_f subalgebra of the right nucleus}). If $\hat{h}(x)$ is irreducible then ${\rm Nuc}_r(S_f)$ has degree $ms=mln/k$ over $F$ by Theorem \ref{thm:mainII}, therefore comparing the dimensions of the vector spaces we obtain the assertion.
\\ (ii)
If $f(t)  \in F[t]\subset R$  then
 $N[t;\sigma]/N[t;\sigma]f$
 has dimension $m[N:F]$ over $F$ and is a subalgebra of ${\rm Nuc}_r(S_f)$ by Theorem \ref{L_f subalgebra of the right nucleus}.
   Suppose that $\hat{h}$ is irreducible, then ${\rm Nuc}_r(S_f)$ has dimension $\frac{mnl}{k}$ as an $F$-vector space (Theorem \ref{thm:mainII}).
   In particular, this implies $\frac{mnl}{k}=[{\rm Nuc}_r(S_f):F]\geq m[N:F]$, a contradiction if  $[N:F]>\frac{nl}{k}$.
\end{proof}

As a direct consequence of Proposition \ref{t^s in Nuc}, we obtain:

\begin{theorem}\label{thm:main6}
Suppose that $f(t) = t^m-\sum_{i=0}^{m-1}a_it^i \in {\rm Fix}(\sigma^c)[t;\sigma]$ for some minimal $c \in \lbrace 1,2,\dots,m-1 \rbrace$. Suppose that $f$ is not right invariant and that $\hat{h}(x)$ is irreducible in $F[x]$.
\\ (i) If $m=qc$ for some positive integer $q$ and $[N:F]>\frac{cnl}{k}$, then $f$ is reducible.
\\ (ii) If $m = qc + r$ for some positive integers $q,r$ with $0 < r < c$, and $[N:F]\geq \frac{cnl}{k}$
then $f$ is reducible.
\end{theorem}

\begin{proof}
Since $\hat{h}$ is irreducible in $F[x]$, then  the right nucleus has dimension $\frac{mnl}{k}$ as an $F$-vector space.
 \\ (i) If $m=qc$ for some positive integer $q$, then
$$N \oplus Nt^c \oplus Nt^{2c} \oplus \dots \oplus Nt^{(q-1)c}$$ is an $F$-sub vector space of ${\rm Nuc}_r(S_f)$ of dimension $q[N:F]$.
\\ (ii) If $m = qc + r$ for some positive integers $q,r$ with $0 < r < c$, then
$$N \oplus Nt^c \oplus Nt^{2c} \oplus \dots \oplus Nt^{qc} $$
is an $F$-sub vector space of ${\rm Nuc}_r(S_f)$ of dimension $(q+1)[N:F]$.
 \\
 If $[N:F]>\frac{cnl}{k}$ in (i), then $q[N:F]>\frac{mnl}{k}$, a contradiction.
  If $[N:F]\geq \frac{cnl}{k}$ in (ii), then $(q+1)[N:F]\geq q\frac{cnl}{k}+\frac{cnl}{k}>\frac{mnl}{k}$, a contradiction. Thus
   $\hat{h}$ must be reducible, and therefore $f$, too.
\end{proof}

%%%%%%%%%%%%%%%%%%%%%%%%%%%%%%%%%%%%%%%%%%%%%%%%%%%%%%%%%%%%%%%%%%%%%%%%%%%%%%%%%%%
%
% The right nucleus of $\mathbf{S_f}$ for low degree polynomials
%
%%%%%%%%%%%%%%%%%%%%%%%%%%%%%%%%%%%%%%%%%%%%%%%%%%%%%%%%%%%%%%%%%%%%%%%%%%%%%%%%%%%%

%%%%%%%%%%%%%%%%%%%%%%%%%%%%%%%%%%%%%%%%%%%%%%%%%%%%%%%%%%%%%%%%%%%%%%%%%%%%%

\section{The right nucleus of $S_f$ for low degree polynomials in $F[t]\subset K[t;\sigma]$}\label{sec:low}

%%%%%%%%%%%%%%%%%%%%%%%%%%%%%%%%%%%%%%%%%%%%%%%%%%%%%%%%%%%%%%%%%%%%%%%%%%%%%

We assume that $K/F$ is a cyclic Galois field extension of degree $n$  with ${\rm Gal}(K/F)= \langle \sigma \rangle$.
We now explore the structure of ${\rm Nuc}_r(S_f)$ for $f\in F[t]\subset R$ of low degree (the same arguments apply for higher degrees).
 We repeatedly use that  $ [{\rm Fix}(\sigma^s):F]  = {\rm gcd}(n,s)$, $N = \bigcap\limits_{\lambda_j \in \Lambda}{\rm Fix}(\sigma^{m-\lambda_j})$ by Theorem \ref{Coeffs} and Corollary \ref{Lsigma,f for a field extension}.
We also use that if $f \in F[t]\subset R$  then $  N[t;\sigma]/N[t;\sigma]f$
is a subalgebra of ${\rm Nuc}_r(S_f)$  (Theorem \ref{L_f subalgebra of the right nucleus}).

%%%%%%%%%%%%%%%%%%%%%%%%%%%%%%%%%%%%%%%%%%%%%%%%%%%%%%%%%%%%%%%%%%%%%%%%%%%%%

\subsection{$m=2$} \label{Example m=2}

%%%%%%%%%%%%%%%%%%%%%%%%%%%%%%%%%%%%%%%%%%%%%%%%%%%%%%%%%%%%%%%%%%%%%%%%%%%%%

Let $f(t)=t^2-a_1t-a_0 \in R$, then
$N = \bigcap\limits_{\lambda_j \in \Lambda}{\rm Fix}(\sigma^{2-\lambda_j})$.

\begin{enumerate}
\item If $f(t)=t^2-a_0$ with $a_0 \in K^\times$, then  $N = {\rm Fix}(\sigma^2)$.
\item If $f(t)=t^2-a_1t-a_0$ with $a_1\in K^\times$, then  $N = F $.
\end{enumerate}

 Note that if $n $ is even, then $\sigma^2$ has order $\frac{n}{2}$ in ${\rm Gal}(K/F)$, which means that $F \neq {\rm Fix}(\sigma^2)$. If $n$ is odd, then ${\rm gcd}(n,2)=1$, therefore ${\rm Fix}(\sigma^2) = F$.

\begin{proposition}
Let $f(t)=t^2-a_0 \in F[t] \subset R$, $a_0\not=0$ then
${\rm Fix}(\sigma^2)[t;\sigma]/{\rm Fix}(\sigma^2)[t;\sigma]f$
 is a subalgebra of ${\rm Nuc}_r(S_f)$ of dimension $2[{\rm Fix}(\sigma^2):F]$ over $F$.
 In particular, if $n$ is prime or odd,  then $f$ is reducible.
\begin{proof}
$N[t;\sigma]/N[t;\sigma]f$
 is a subalgebra of ${\rm Nuc}_r(S_f)$ and
 $N = {\rm Fix}(\sigma^2)$ by (1), which yields the first assertion.
The second assertion follows from the fact that the right nucleus has dimension 2 over $F$ for irreducible right invariant $f$ under our assumptions.
\end{proof}
\end{proposition}

%%%%%%%%%%%%%%%%%%%%%%%%%%%%%%%%%%%%%%%%%%%%%%%%%%%%%%%%%%%%%%%%%%%%%%%%%%%%%

\subsection{$m=3$} \label{Example m=3}

%%%%%%%%%%%%%%%%%%%%%%%%%%%%%%%%%%%%%%%%%%%%%%%%%%%%%%%%%%%%%%%%%%%%%%%%%%%%%

Let $f(t) \in R$ be of degree $3$, then
 $N = \bigcap\limits_{\lambda_j \in \Lambda}{\rm Fix}(\sigma^{3-\lambda_j}) $.

\begin{enumerate}
\item If $f(t) = t^3-a_0 \in R$, where $a_0 \in K^\times$, then  $N = {\rm Fix}(\sigma^3) $.
\item If $f(t) = t^3-a_1t \in R$, where $a_1 \in K^\times$, then  $N =  {\rm Fix}(\sigma^2)$.
\item In all other cases,  $N = F$.
\end{enumerate}

\begin{proposition}
(i) If $f(t)=t^3-a_0$ with $0\not=a_0 \in F$, then
$${\rm Fix}(\sigma^3)[t;\sigma]/{\rm Fix}(\sigma^3)[t;\sigma] f$$
is a subalgebra of ${\rm Nuc}_r(S_f)$ of dimension $3[{\rm Fix}(\sigma^3):F]$ over $F$.
In particular, if $n$ is prime or not divisible by 3,  %then ${\rm Fix}(\sigma^3)\not=F$
then $f$ is reducible.
\\ (ii) If $f(t)=t^3-a_1t $ with $0\not=a_1 \in F$, then
$${\rm Fix}(\sigma^2)[t;\sigma]/{\rm Fix}(\sigma^2)[t;\sigma]f$$
is a subalgebra of ${\rm Nuc}_r(S_f)$ of dimension $3[{\rm Fix}(\sigma^2):F]$ over $F$.
\end{proposition}

\begin{proof}
By  Proposition \ref{L_f subalgebra of the right nucleus}, $N[t;\sigma]/N[t;\sigma]f(t)\subset {\rm Nuc}_r(S_f)$.
\\ (i)  If $f(t)=t^3-a_0 \in F[t]$ with $a_0 \neq 0$, then $N = {\rm Fix}(\sigma^3)$  which proves the assertion looking at the dimensions.
 \\ (ii) If $f(t)=t^3-a_1t \in F[t]$ with $a_1 \neq 0$, then $N = {\rm Fix}(\sigma^2)$.
\end{proof}

%%%%%%%%%%%%%%%%%%%%%%%%%%%%%%%%%%%%%%%%%%%%%%%%%%%%%%%%%%%%%%%%%%%%%%%%%%%%%

\subsection{$m=4$} \label{Example m=4}

%%%%%%%%%%%%%%%%%%%%%%%%%%%%%%%%%%%%%%%%%%%%%%%%%%%%%%%%%%%%%%%%%%%%%%%%%%%%%

Let $f(t) \in R$ be of degree $4$, then
 $N = \bigcap\limits_{\lambda_j \in \Lambda} {\rm Fix}(\sigma^{4-\lambda_j})  $.

\begin{enumerate}
\item If $f(t)=t^4-a_0$ with $a_0 \in K^\times$ then  $N = {\rm Fix}(\sigma^4)$.
\item If $f(t)=t^4-a_1t$ with $a_1 \in K^\times$ then   $N = {\rm Fix}(\sigma^3) $.
\item If $f(t)=t^4-a_2t^2$ with $a_2 \in K^\times$ then  $N = {\rm Fix}(\sigma^2) $.
\item If $f(t)=t^4-a_2t^2-a_0$ with $a_0,a_2 \in K^\times$, then  $N  = {\rm Fix}(\sigma^4) \cap {\rm Fix}(\sigma^2) = {\rm Fix}(\sigma^2)$.
\item In all other cases, $N = {\rm Fix}(\sigma) = F $.
\end{enumerate}

Observe that:
\begin{itemize}
\item If $n \equiv 0 (mod \text{ } 4)$, then $[{\rm Fix}(\sigma^4):F]=4$.
\item If $n \equiv 1 \text{ or } 3 (mod \text{ } 4)$, then ${\rm Fix}(\sigma^4)=F$.
\item If $n \equiv 2 (mod \text{ } 4)$, then $[{\rm Fix}(\sigma^4):F] =2$.
\item If $n \equiv 0 (mod \text{ } 3)$, then $[{\rm Fix}(\sigma^3):F]=3$.
\item If $n \equiv 1 \text{ or } 2 (mod \text{ } 3)$, then ${\rm Fix}(\sigma^3)=F$.
\item If $n \equiv 0 (mod \text{ } 2)$ then $[{\rm Fix}(\sigma^2):F] = 2$.
\item If $n \equiv 1 (mod \text{ } 2)$ then ${\rm Fix}(\sigma^2)=F$.
\end{itemize}

\begin{proposition}
(i) If $f(t)=t^4-a_0\in F[t]$ with $0\not=a_0 $, then
$${\rm Fix}(\sigma^4)[t;\sigma]/{\rm Fix}(\sigma^4)[t;\sigma]f$$
is a subalgebra of ${\rm Nuc}_r(S_f)$
of dimension ${\rm gcd}(n,4)$ over $F$.
In particular:
\\ (a) If $f$ is irreducible and either $n\not=2$ is prime or $gcd(n,4)=1$,
then
$${\rm Nuc}_r(S_f) \cong {\rm Fix}(\sigma^4)[t;\sigma]/{\rm Fix}(\sigma^4)[t;\sigma]f.$$
 (b) If   $n=2$, then $f$ is reducible.
\\ (ii) If $f(t)=t^4-a_1t\in F[t;\sigma]$ with $0\not= a_1$, then
$${\rm Fix}(\sigma^3)[t;\sigma]/{\rm Fix}(\sigma^3)[t;\sigma]f$$
is a subalgebra
of ${\rm Nuc}_r(S_f)$ of dimension $4 {\rm gcd}(n,3)$ over $F$.
\\ (iii) If $f(t)=t^4-a_2t-a_0\in F[t;\sigma]$ with $0\not= a_2 $, then
$${\rm Fix}(\sigma^2)[t;\sigma]/{\rm Fix}(\sigma^2)[t;\sigma]f$$
is a subalgebra of ${\rm Nuc}_r(S_f)$ of dimension $4 {\rm gcd}(n,2)$ over $F$.
In particular:
\\ (a) If $f$ is irreducible,  and either $n\not=2$ is prime or ${\rm gcd}(n,4)=1$,
then
$${\rm Nuc}_r(S_f) \cong {\rm Fix}(\sigma^2)[t;\sigma]/{\rm Fix}(\sigma^2)[t;\sigma]f.$$
 (b) If  $n=2$, then $f$ is reducible.
\end{proposition}

\begin{proof}
 $N[t;\sigma]/N[t;\sigma]f(t)\subset {\rm Nuc}_r(S_f)$ by Theorem \ref{L_f subalgebra of the right nucleus}.
\\ (i) Here $N = {\rm Fix}(\sigma^4)$ by (1), and thus $$ {\rm Fix}(\sigma^4)[t;\sigma]/{\rm Fix}(\sigma^4)[t;\sigma]f\subset {\rm Nuc}_r(S_f).$$
 (ii) We know $N = {\rm Fix}(\sigma^3)$ by (2), and hence
$$ {\rm Fix}(\sigma^3)[t;\sigma]/{\rm Fix}(\sigma^3)[t;\sigma]f\subset {\rm Nuc}_r(S_f).$$
 (iii) We have $N = {\rm Fix}(\sigma^2)$ by (3), and so
$$ {\rm Fix}(\sigma^2)[t;\sigma]/{\rm Fix}(\sigma^2)[t;\sigma]f\subset {\rm Nuc}_r(S_f).$$
\end{proof}

%%%%%%%%%%%%%%%%%%%%%%%%%%%%%%%%%%%%%%%%%%%%%%%%%%%%%%%%%%%%%%%%%%%%%%%%%%%
%
% Algorithm
%
%%%%%%%%%%%%%%%%%%%%%%%%%%%%%%%%%%%%%%%%%%%%%%%%%%%%%%%%%%%%%%%%%%%%%%%%%%%%%

\section{A small algorithm to check if $f$ is reducible}\label{sec:conclusion}

Let $K/F$ be a cyclic Galois extension of degree $n$ with Galois group ${\rm Gal}(K/F)=\langle \sigma \rangle$.
We assume  that $n$ is either prime or that ${\rm gcd}(m,n)=1$ to simplify the process.
For some skew polynomials $f(t) = t^m-\sum_{i=0}^{m-1}a_it^i\in R$ 
which are not right invariant, we can decide if they are reducible based on the following ``algorithm'' with output \fbox{TRUE} if $f$ is reducible and \fbox{STOP} if we cannot decide:
\begin{enumerate}
\item Check if $f\in F[t]$.
If $f$ is reducible in $F[t]$, then $f$ is reducible in $R$ \fbox{TRUE}.
If $f\not \in F[t]$ then go to (2).
\item Compute $N={\rm Fix}(\sigma^d)$, where $d = {\rm gcd}(m-\lambda_1,m-\lambda_2,\dots,m-\lambda_r,n)$ as per Theorem \ref{ConditionforL=F}.
\\ If $[N:F]>m$, then $f$ is reducible \fbox{TRUE}.
\\ If $[N:F]\leq m$ then go to (3).
\item Find the smallest integer $c$, such that $a_i\in {\rm Fix}(\sigma^c)$ for all $i$, and where ${\rm Fix}(\sigma^c)$
is a proper subfield of $K$.
\\ If ${\rm Fix}(\sigma^c)=N$ then $f$ is reducible  \fbox{TRUE}.
\\ If $m=qc$ and $[N:F]>c$, then $f$ is reducible  \fbox{TRUE}.
\\  If $m = qc + r$ with $0 < r < c$, and $[N:F]\geq c$
then $f$ is reducible \fbox{TRUE}.
\\ In all other cases, go to (4).
\item If all $a_i$ are not contained in a proper subfield of $K$, then we \emph{cannot decide} if $f$ is reducible \fbox{STOP}.
\end{enumerate}

Furthermore, if $f(t) \in F[t]$ then we can use the fact that
$ N[t;\sigma]/N[t;\sigma]f$
is a subalgebra of ${\rm Nuc}_r(S_f)$ to look for zero divisors in ${\rm Nuc}_r(S_f)$ in order to factor $f$.

%*******************************************************************************************%
%****************************************************************************************%

\end{document}